\RequirePackage{fix-cm}
\documentclass[smallextended]{svjour3}       
\smartqed  
\usepackage{graphicx}
\usepackage{amsthm}
\usepackage{amsmath}
\usepackage{amssymb}
\usepackage{mathrsfs}
\usepackage{amsfonts}
\usepackage{algorithm}
\usepackage{algorithmicx}
\usepackage{algpseudocode}
\usepackage{cases}
\usepackage{enumerate}

\begin{document}

\title{Newton-based methods for finding the positive ground state of Gross-Pitaevskii equations\thanks{\textbf{Funding. }This work was funded by the National Natural Science Foundation of China (Grant No. 11671217, No. 12071234) and the Tianjin Graduate Research and Innovation Project (No. 2019YJSB040).}
}

\titlerunning{Newton-based methods for finding the positive ground state of GPE}        

\author{Pengfei Huang         \and
        Qingzhi Yang 
}


\institute{Pengfei Huang \at
              School of Mathematical
Sciences and LPMC, Nankai University, Tianjin 300071, P.R. China. \\
              \email{huangpf@mail.nankai.edu.cn}. ORCID: 0000-0003-3097-1804.           
           \and
           Corresponding author. Qingzhi Yang \at
              School of Mathematical
Sciences and LPMC, Nankai University, Tianjin 300071, P.R. China. \\
              \email{qz-yang@nankai.edu.cn}
}

\date{Received: date / Accepted: date}

\maketitle

\begin{abstract}
The discretization of Gross-Pitaevskii equations (GPE) leads to a nonlinear eigenvalue problem with eigenvector nonlinearity (NEPv). In this paper, we use two Newton-based methods to compute the positive ground state of GPE. The first method comes from the Newton-Noda iteration for saturable nonlinear Schr\"odinger equations proposed by Liu, which can be transferred to GPE naturally. The second method combines the idea of the Bisection method and the idea of Newton method, in which, each subproblem involving block tridiagonal linear systems can be solved easily. We give an explicit convergence and computational complexity analysis for it. Numerical experiments are provided to support the theoretical results.
\keywords{Gross-Pitaevskii equations \and Nonlinear eigenvalue \and Ground state \and Newton method \and Bisection method}
\subclass{65H17 \and 49M15 \and 65N25}
\end{abstract}

\section{Introduction}
The Gross-Pitaevskii equation (GPE) is a nonlinear Schr\"odinger equation that describes the properties of condensate at zero or very low temperature. Such an equation for the non-rotating Bose-Einstein condensation (BEC) can be expressed as \cite{bao2013mathematical,pethick2008bose}
\begin{equation}\label{equ:GPE}
i\frac{\partial\phi(x,t)}{\partial t}=(-\Delta+V(x)+\beta|\phi(x,t)|^2)\phi(x,t),\quad x\in\mathbb{R}^d,~t>0,
\end{equation}
where $\phi(x,t)$ is a wave function, $d=1,2,3$, $\beta\in\mathbb{R}$ and $V(x)$ is a real-valued external trapping potential. BEC is of interest in many applications \cite{pethick2008bose,fetter2009rotating}, and one of the major topics is finding the ground state of \eqref{equ:GPE}. The ground state of equation \eqref{equ:GPE} is defined as the minimizer of the energy functional, which can be expressed as the following nonconvex minimization problem \cite{bao2013mathematical,bao2004computing,bao2003ground}:
\begin{equation}\label{equ:groundstate}
u_g=\arg\min_{u\in S} E(u),
\end{equation}
where
\begin{equation*}\label{equ:energy}
E(u)=\int_{\mathbb{R}^d}[|\nabla u(x)|^2+V(x)|u(x)|^2+\frac{\beta}{2}|u(x)|^4]dx,
\end{equation*}
and the spherical constraint $S$ is defined as
\begin{equation*}\label{equ:spherical}
S=\{u|E(u)<\infty,\int_{\mathbb{R}^d}|u(x)|^2dx=1\}.
\end{equation*}
The Euler-Lagrange equation (or first-order optimality condition) of \eqref{equ:groundstate} is the nonlinear eigenvalue problem as follows:
\begin{equation}\label{equ:NEPV1}
-\Delta u+V(x)u+\beta|u|^2u=\lambda u,~\int_{\mathbb{R}^d}|u(x)|^2dx=1,
\end{equation}
where $(\lambda,u)$ is the eigenpair. In this paper, we consider the finite-difference discretization of the nonlinear eigenvalue problem \eqref{equ:NEPV1} with $\beta>0$. This leads to a nonlinear eigenvalue problem with eigenvector nonlinearity (NEPv)
\begin{equation}\label{equ:NEPv}
\beta \mathrm{diag}(u^{[2]})u+Bu=\lambda u,~u^Tu=1,
\end{equation}
where $B\in\mathbb{R}^{n\times n}$ is the sum of the discretization matrix of the negative Laplace operator and the external trapping potential. Here we still use $u=[u_1,u_2,\cdots,u_n]^T$ to denote the discretization of $u(x)$ without ambiguity, $u^{[2]}=[u_1^2,u_2^2,\cdots,u_n^2]^T$, and $\mathrm{diag}(u)$ represents a diagonal matrix with the diagonal given by the vector $u$. It is well known that, without rotation, the ground state can be seen as a real non-negative function in physics \cite{bao2013mathematical}. Furthermore, the positiveness of the eigenvector of \eqref{equ:NEPv}, corresponding to the discretization of the ground state, has been proved from different aspects, see \cite{cances2010numerical,choi2001generalization,huang2020finding}. Therefore, we aim to provide methods for finding the positive solution of \eqref{equ:NEPv} and analysis convergence and computational complexity.

Different methods have been proposed in computing the ground state of BEC for nonlinear eigenvalue problems \eqref{equ:NEPV1} and the minimization problem \eqref{equ:groundstate}, such as the self-consistent field iteration (SCF) \cite{cai2018eigenvector}, full multigrid method \cite{jia2016full}, normalized gradient method \cite{bao2004computing}, regularized Newton method \cite{wu2017regularized,hu2018adaptive} and semidefinite programming relaxation method \cite{hu2016note,yang2019numerical}. To preserve the positivity of the solution of \eqref{equ:NEPv}, we refer to two different ideas based on the Newton method for nonlinear equations. Recently, Liu \cite{liu2020positivity} proposed the Newton-Noda iteration (NNI) to compute the positive solution of the saturable nonlinear Schr\"odinger equations, combining the idea of Newton method with the idea of the Noda iteration \cite{noda1971note} and proved its convergence. It is easy to find that the NNI is also effective for our problem and enjoys locally quadratic convergence. Since the cost of NNI is dominated by inner linear systems, while given $u$, $\beta \mathrm{diag}(u^{[2]})+B$ in \eqref{equ:NEPv} itself has nice structure as block tridiagonal positive semidefinite $M$-matrix, we want to solve \eqref{equ:NEPv} without the spherical constraint in the iteration. For a given proper $\lambda>0$, Choi et al. \cite{choi2001generalization,choi2002global} used the Newton iteration to find the positive eigenvector of the unconstrained NEPv,
\begin{equation*}
\beta \mathrm{diag}(u^{[2]})u+Bu=\lambda u,
\end{equation*}
and proved the global monotone convergence of the Newton iteration. They also asserted that the positive eigenvector is differentiable about the eigenvalue $\lambda$. This motivated us to propose a simple algorithm combining the idea of Newton method with the idea of the Bisection method to obtain an acceptable approximate solution of \eqref{equ:NEPv}. We present the convergence and computational complexity analysis for our algorithm.

The rest of this paper is organized as follows. We begin with some preliminaries in section \ref{sec:pre}. In section \ref{sec:NNI}, we present the Newton-Noda iteration and its convergence theories for NEPv \eqref{equ:NEPv}. In section \ref{sec:NBI}, we propose the Newton-Bisection iteration and establish the convergence of the algorithm and give its computational complexity. In section \ref{sec:numerical}, we provide numerical examples of Newton-Noda and Newton-Bisection iteration. Concluding remarks are given in the last section.
\section{Preliminaries}\label{sec:pre}
Throughout this paper, we adopt the standard linear algebra notations. $\|\cdot\|$ denotes the 2-norm of vectors and matrices and $\|u\|_1$ is the 1-norm for the vector $u$. In addition, for a vector $u=[u_1,u_2,\cdots,u_n]^T$,
\[\min(u)=\min\limits_iu_i,\quad \max(u)=\max\limits_iu_i.\]
For $u_1,~u_2\in\mathbb{R}^n$, $u_1>(\ge)u_2$ denotes that $(u_1)_i>(\ge)(u_2)_i$, $i=1,2,\cdots, n$.
$\lambda_{max}(B)$  and $\lambda_{min}(B)$ denote the maximum and the minimum eigenvalues of the matrix $B$, respectively.
\begin{definition}[$M$-matrix \cite{varga1962iterative}]
A matrix $B\in\mathbb{R}^{n\times n}$ is called an $M$-matrix, if $B=sI-A$, where $A$ is nonnegative, and $s\ge\rho(A)$. Here $\rho(A)$ is the spectral radius of $A$.
\end{definition}
\begin{definition}[Irreducibility/Reducibility \cite{varga1962iterative}]
A matrix $B\in \mathbb{R}^{n\times n}$ is called reducible, if there exists a nonempty proper index subset $I\subset[n]$, such that
$$b_{ij}=0,\quad \forall i\in I,\quad\forall j\notin I.$$
If $B$ is not reducible, then we call $B$ irreducible.
\end{definition}

Here, we consider NEPv \eqref{equ:NEPv} generated by the finite difference discretization of \eqref{equ:NEPV1}, thus $B$ is an irreducible nonsingular $M$-matrix.
Let us define several notations that will be used throughout this paper, which is similar to \cite{liu2020positivity}:
\[\mathcal{A}(u)=\beta \text{diag} (u^{[2]})+B,\quad r(u,\lambda)=\mathcal{A}(u)u-\lambda u.\]
Then NEPv \eqref{equ:NEPv} can be written as
\[r(u,\lambda)=\mathcal{A}(u)u-\lambda u=0,~u^Tu=1,\]
and $J(u,\lambda)=\nabla_ur(u,\lambda)=\mathcal{A}(u)-\lambda I+2\beta\mathrm{diag}(u^{[2]})$.
\section{The Newton-Noda Iteration for NEPv and its inexact version}\label{sec:NNI}
In this section, we present the Newton-Noda iteration (NNI) for computing the positive solution of NEPv \eqref{equ:NEPv} in the way analogous to Liu \cite{liu2020positivity}. Define
\[F(u,\lambda)=
\begin{bmatrix}
r(u,\lambda)\\
\frac{1}{2}(1-u^Tu)
\end{bmatrix},
\]
then
\[F'(u,\lambda)=
\begin{bmatrix}
J(u,\lambda) & -u\\
-u^T&0
\end{bmatrix}.
\]
The NNI is showed in Algorithm \ref{alg:NNI}.
\begin{algorithm}[htbp]
\caption{Newton-Noda iteration (NNI) for NEPv} \label{alg:NNI}
\begin{algorithmic}[1]
\State Given a feasible initial point $u_0>0$ with $\|u_0\|=1$, $\lambda_0=\text{min}(\frac{\mathcal{A}(u_0)u_0}{u_0})$.
    \For{$k=0,1,2,\cdots$}
        \State Solve the linear system $F'(u_k,\lambda_k)[\begin{smallmatrix}\Delta_k\\ \delta_k\end{smallmatrix}]=-F(u_k,\lambda_k)$.
        \State Let $\theta_k=1$.
        \State Compute $w_{k+1}=u_k+\theta_k\Delta_k$.
        \State Normalize the vector $\hat{u}_{k+1}=\frac{w_{k+1}}{\|w_{k+1}\|}$.
        \State Compute $h_{k}(\theta_k)=\mathcal{A}(\hat{u}_{k+1})\hat{u}_{k+1}-\lambda_k\hat{u}_{k+1}$.
        \While{$h_k(\theta_k)\ngtr 0$}
            \State $\theta_k=\frac{\theta_k}{2}$, go back to step 5.
        \EndWhile
        \State $u_{k+1}=\hat{u}_{k+1}$, compute $\lambda_{k+1}=\text{min}(\frac{\mathcal{A}(u_{k+1})u_{k+1}}{u_{k+1}})$.
    \EndFor
\end{algorithmic}
\end{algorithm}

Suppose that $\{u_k,\lambda_k\}$ is generated by Algorithm \ref{alg:NNI}. We can obtain the following convergence results as an extension of theories of Liu \cite{liu2020positivity}, which we will not repeat the proof in here.
\begin{theorem}
The sequence $\{\lambda_k\}$ is strictly increasing and bounded above. The limit point of $\{u_k\}$ is a positive eigenvector of NEPv \eqref{equ:NEPv}.
\end{theorem}
\begin{theorem}
Suppose that $\{u_0,\lambda_0\}$ is sufficiently close to an positive eigenpair $(u_*,\lambda_*)$ of NEPv \eqref{equ:NEPv}. Then $\lambda_k$ and $u_k$ converge quadratically to $\lambda_*$ and $u_*$, respectively.
\end{theorem}

It is obvious that the main step in Algorithm \ref{alg:NNI} is to solve the linear system of step 3. In our numerical experiments, we only solve it inexactly by iterative method for efficiency in large scale, although we cannot verify the convergence of the inexact version theoretically.

\begin{algorithm}
\caption{Inexact Newton-Noda iteration (NNI) for NEPv} \label{alg:inexactNNI}
\begin{algorithmic}[1]
\State Given a feasible initial point $u_0>0$ with $\|u_0\|=1$, $\lambda_0=\text{min}(\frac{\mathcal{A}(u_0)u_0}{u_0})$. $c>0$.
    \For{$k=0,1,2,\cdots$}
        \State Solve the linear system inexactly, such that $\|F'(u_k,\lambda_k)[\begin{smallmatrix}\Delta_k\\ \delta_k\end{smallmatrix}]+F(u_k,\lambda_k)\|<c.$
        \State Let $\theta_k=1$.
        \State Compute $w_{k+1}=u_k+\theta_k\Delta_k$.
        \State Normalize the vector $\hat{u}_{k+1}=\frac{w_{k+1}}{\|w_{k+1}\|}$.
        \State Compute $h_{k}(\theta_k)=\mathcal{A}(\hat{u}_{k+1})\hat{u}_{k+1}-\lambda_k\hat{u}_{k+1}$.
        \While{$h_k(\theta_k)\ngtr 0$}
            \State $\theta_k=\frac{\theta_k}{2}$, go back to step 5.
        \EndWhile
        \State $u_{k+1}=\hat{u}_{k+1}$, compute $\lambda_{k+1}=\text{min}(\frac{\mathcal{A}(u_{k+1})u_{k+1}}{u_{k+1}})$.
    \EndFor
\end{algorithmic}
\end{algorithm}

In the next section, we propose a method to solve the linear system only involves $J(u,\lambda)$, which has nice structure and there are mature methods to compute it directly or iteratively, so as to reduce the workload of the algorithm.
\section{The Newton-Bisection Iteration}\label{sec:NBI}
In \cite{choi2001generalization,choi2002global} Choi et al. used the Newton iteration to solve the following unconstrained nonlinear eigenvalue problem:
\begin{equation}\label{equ:choi}
\beta \text{diag}(u^{[2]})u+Bu=\lambda u,
\end{equation}
where $\lambda$ is any fixed positive constant such that $\lambda>\lambda_{min}(B)$. The Newton iteration for solving \eqref{equ:choi} is:
\begin{equation}\label{equ:iter}
J(u_k,\lambda)u_{k+1}=2\beta \text{diag}(u_k^{[2]})u_k,~k=1,2,\cdots.
\end{equation}
\begin{lemma}\cite[Theorem1, Theorem 2]{choi2002global}\label{lem:continuous}
For any $\lambda>\mu$, where $\mu=\lambda_{min}(B)$, \eqref{equ:choi} has a unique positive solution. Let $u(\lambda)$ denote the unique positive eigenvector corresponding to $\lambda\in(\mu,\infty)$. Then:
\begin{enumerate}[(1)]
\item $u(\lambda_1)<u(\lambda_2)$ if $\mu<\lambda_1<\lambda_2<\infty$;
\item $u(\lambda)$ is continuous on $(\mu,\infty)$.
\end{enumerate}
\end{lemma}

The convergence properties of Newton iteration for \eqref{equ:iter} are formulated from \cite{choi2002global} in the following theorem.
\begin{theorem}\label{thm:newton}
Let $p$ be a positive eigenvector of $B$ corresponding to $\lambda_{min}(B)$. Let $u$ be the unique positive solution of \eqref{equ:choi} for some $\lambda>\lambda_{min}(B)$. Let $u_0=\alpha p$, where $\alpha$ is large enough such that
\begin{equation}\label{equ:condition}
u_0>u,\quad \underset{1\le i\le n}{\text{min}}\beta (\alpha p_i)^2>\lambda-\lambda_{min}(B).
\end{equation}
Then the iteration \eqref{equ:iter} converges to $u$ monotonically, $u<\cdots<u_2<u_1<u_0$. When $u_k$ is sufficiently close to $u$, \eqref{equ:iter} is locally quadratically convergent.
\end{theorem}

Based on Theorem \ref{thm:newton} and Lemma \ref{lem:continuous}, we propose an algorithm, called the Newton-Bisection iteration (NBI), to solve the positive eigenvector of NEPv \eqref{equ:NEPv}.
\begin{algorithm}[H]
\caption{Newton-Bisection Iteration (NBI)}\label{alg:NBI}
\begin{algorithmic}[1]
\State Given $[a,b]$ the initial interval for $\lambda$ with $a>\mu$. Given initial point $u_0$ and $tol>0$.
\For{$k=1,2,\cdots$}
    \State Let $\lambda_k=\frac{a+b}{2}$ and $u_k^0=u_0$.
    \State (Solve $\beta\text{diag}(u^{[2]})u+Bu=\lambda_k u$ by Newton iteration, i.e. step 5.-7.)
    \For{$l=0,1,\cdots$}
        \State Solve the linear system $[3\beta\text{diag}((u_k^l)^{[2]})+B-\lambda_k I]u=2\beta\text{diag}((u_k^l)^{[2]})u_k^l$.
    \EndFor
    \State $u_k=u_k^l$. If $|\|u_k\|-1|<tol$, stop. $b=\frac{a+b}{2}$ if $\|u_k\|>1$, $a=\frac{a+b}{2}$ otherwise.
\EndFor
\end{algorithmic}
\end{algorithm}
\subsection{Complexity analysis of NBI}
In this subsection, we give the detailed convergence and computational complexity analysis for NBI. Let $[a,b]$ be the initial interval for the bisection iteration. We always assume that $\lambda_{min}(B)<a<\lambda_*<b$, where $\lambda_*$ is the positive eigenvalue for \eqref{equ:NEPv}, and $u_0$ is sufficiently large to satisfy \eqref{equ:condition} for $\lambda=b$. First, we offer the convergence rate for the outer bisection iteration.
\begin{lemma}\label{lem:inverse}
If $(u,\lambda)$ is a positive eigenpair of NEPv \eqref{equ:NEPv}, then $J(u,\lambda)$ is a nonsingular $M$-matrix and is invertible.
\end{lemma}
\begin{proof}
Since $\mathcal{A}(u)=\beta\text{diag}(u^{[2]})+B$ is an irreducible nonsingular $M$-matrix for any $u>0$ and $\beta>0$, $\lambda$ is indeed the smallest eigenvalue of $\mathcal{A}(u)$. Then, $\mathcal{A}(u)-\lambda I$ is a positive semidefinite $M$-matrix. Thus, $J(u,\lambda)=\mathcal{A}(u)-\lambda I+2\beta\mathrm{diag}(u^{[2]})$ is a nonsingular $M$-matrix.
\end{proof}
\begin{theorem}\label{thm:bisection}
Suppose $(u_*,\lambda_*)$ is a positive eigenpair of NEPv \eqref{equ:NEPv}, $\{(u_k,\lambda_k)\}$ is the sequence generated in the bisection iteration of NBI. Then
\begin{equation*}\label{equ:bisection}
\|u_k-u_*\|\le \frac{M(b-a)}{2^k},
\end{equation*}
where $M=\|u_0\|/({2\beta\min(u_*^{[2]})})$.
\end{theorem}
\begin{proof}
According to Lemma \ref{lem:continuous}, $u(\lambda)$ is continuous on $(\mu,\infty)$, for each $(u(\lambda),\lambda)$ satisfying
\[r(u(\lambda),\lambda)=\mathcal{A}(u(\lambda))u(\lambda)-\lambda u(\lambda)=0.\]
Combined with Lemma \ref{lem:inverse}, $J(u,\lambda)$ is invertible, then we have $\partial u(\lambda)/\partial \lambda=J(u,\lambda)^{-1}u$. For any $u$ that satisfies $u(a)<u<u_0$, we also have
\begin{equation*}
\begin{aligned}
\|J(u,\lambda)^{-1}u\|&=\|(3\beta\text{diag}(u(\lambda)^{[2]})+B-\lambda I)^{-1}u\|\\
&\le\frac{\|u_0\|}{\lambda_{min}(2\beta\text{diag}(u^{[2]})+\mathcal{A}(u)-\lambda I)}\\
&\le \frac{\|u_0\|}{2\beta\text{min}(u^{[2]})+\lambda_{min}(\mathcal{A}(u)-\lambda I)}\\
&\le \frac{\|u_0\|}{2\beta\text{min}(u(a)^{[2]})}.
\end{aligned}
\end{equation*}
According to Lemma \ref{lem:continuous}, $u(a)<u_k<u_0$. Denote $M=\|u_0\|/({2\beta\text{min}(u(a)^{[2]})})$, then for the $k$-th bisection iteration, we have
\begin{equation*}
\|u_k-u_*\|\le M\|\lambda_k-\lambda_*\|\le\frac{M(b-a)}{2^k}.
\end{equation*}
\end{proof}
\begin{remark}\label{rem:bisection}
According to Theorem \ref{thm:bisection}, to reach the stopping condition $|\|u_k\|-1|<tol$, it needs at most $K=\lceil log_2(M(b-a)/tol)\rceil$ steps, since $|\|u_K\|-1|\le \|u_K-u_*\|\le tol$. However, we only solve the nonlinear equation iteratively in the step 4-7, so $u_k$ might not be the real solution for \eqref{equ:choi}. Assume when the Newton iteration solving \eqref{equ:choi} stops, $\|u_k-u_k^*\|\le tol/2$, where $u_k^*$ is the exact solution for \eqref{equ:choi}. Then we have
\[\|u_k-u_*\|\le \|u_k-u_k^*\|+\|u_k^*-u_*\|\le \frac{M(b-a)}{2^k}+\frac{tol}{2}.\]
The outer bisection iteration will be at most $K=\lceil log_2(M(b-a)/tol)\rceil+1$.
\end{remark}
Now, let us look into the Newton iteration for solving \eqref{equ:choi} during each bisection iteration. According to Theorem \ref{thm:newton}, it still satisfies the classical locally quadratic convergence. Suppose $u_k^*$ is the positive eigenvector for \eqref{equ:choi}, corresponding to a given $\lambda_k$ in the $k$-th bisection iteration, and $\{u_k^l\}$ is the sequence generated by the Newton iteration. Denote $r_k(u)=\mathcal{A}(u)u-\lambda_k u$. We intend to give an explicit estimation about when it will fall into the scope of quadratic convergence. Then we can obtain the total number of iterations.
\begin{theorem}\label{thm:inneriter}
Let
$l_0=\lceil(\lambda_{max}(B)+3\beta \max(u_0^{[2]})-\lambda_k)\|u_0-u_k^*\|_1/\eta\rceil,$
where $\eta=4\beta^2\min(u_k^*)^4/M_1$ and $M_1>0$ is a constant determined by $u_0,~u_k^*$.
For any $l>l_0$,
\begin{equation}\label{equ:quadratic}
\|r_k(u_k^{l+1})\|\le \frac{M_1}{8\beta^2\min(u_k^*)^4}\|r_k(u_k^l)\|^2.
\end{equation}
\end{theorem}
\begin{proof}
Suppose $r_k(u_k^l)\ge\eta$ for any $l\le l_0$, we have
\begin{equation*}
\begin{aligned}
\|u_k^{l+1}-u_k^l\|&=\|J(u_k^l,\lambda_k)^{-1}r_k(u_k^l)\|\\
&\ge \frac{1}{\|3\beta\text{diag}((u_k^l)^{[2]})+B-\lambda_kI\|}\|r_k(u_k^l)\|\\
&\ge \frac{\eta}{3\beta\max(u_0^{[2]})+\lambda_{max}(B)-\lambda_k}.
\end{aligned}
\end{equation*}
The last inequality results from the monotonicity of $\{u_k^l\}_{l=1}^{\infty}$, that is, $u_k^*<\cdots<u_k^{l+1}<u_k^l<\cdots<u_0$. The monotonicity then leads to that \begin{equation*}
\|u_0-u_k^*\|_1\ge\|u_0-u_k^{l_0}\|_1=\sum\limits_{l=1}^{l_0}\|u_k^{l-1}-u_k^l\|_1
\ge l_0\frac{\eta}{3\beta\max(u_0^{[2]})+\lambda_{max}(B)-\lambda_k}.
\end{equation*}
Thus, we obtain that $l_0\le \lceil(\lambda_{max}(B)+3\beta \max(u_0^{[2]})-\lambda_k)\|u_0-u_k^*\|_1/\eta\rceil$. On the other hand, for $r_k(u_k^l)<\eta$, combined with $u_k^*<u_k^{l+1}<u_k^l<u_0$, we have
\begin{equation*}
\begin{aligned}
\|r_k(u_k^{l+1})\|& = \|r_k(u_k^{l+1})-(r_k(u_k^l)+\nabla_u r_k(u_k^l)(u_k^{l+1}-u_k^l))\|\\
(\text{Taylor's Theorem})&\le\frac{M_1}{2}\|u_k^{l+1}-u_k^l\|^2\\
&\le \frac{M_1}{8\beta^2\min(u_k^*)^4}\|r_k(u_k^l)\|^2\\
&< \frac{\eta}{2}.
\end{aligned}
\end{equation*}
Thus, for any $l> l_0$, \eqref{equ:quadratic} holds.
\end{proof}
According to \eqref{equ:quadratic}, for $l>l_0$, we have
\begin{equation*}
\begin{aligned}
&\|\frac{M_1}{8\beta^2\min(u_k^*)^4}r_k(u_k^{l+1})\|\le \|\frac{M_1}{8\beta^2\min(u_k^*)^4}r_k(u_k^{l})\|^2\\
\le& \|\frac{M_1}{8\beta^2\min(u_k^*)^4}r_k(u_k^{l_0+1})\|^{2^{l-l_0}}
\le (\frac{1}{2})^{2^{l-l_0}}.
\end{aligned}
\end{equation*}
To reach the precision that $\|r_k(u)\|<\epsilon$, we need at most $\lceil l_0+log_2log_2(1/\epsilon)\rceil$ iterative steps. Since $\nabla r_k(u)=3\beta\text{diag}(u^{[2]})+B-\lambda_k I$ is continuous and invertible for $u$, we also have that $\|u_k^l-u_k^*\|\le c\|r_k(u_k^l)-r_k(u_k^*)\|$ for some $c>0$, whenever $u_k^*<u<u_0$. To obtain the precision in Remark \ref{rem:bisection} that $\|u_k-u_k^*\|<\frac{tol}{2}$, we need at most $\lceil l_0+log_2log_2(2c/tol)\rceil$ iterative steps.
\begin{remark}
According to Lemma \ref{lem:continuous}, $u_k^*$ will be bounded by $u(a)$ and $u_0$ as $u(a)<u_k^*<u_0$. Here $a$ is the left end of the initial interval for bisection.
\begin{remark}\label{rem:initial}
In practice, we want to explore the information in the last bisection step. That is to use $u_{k-1}$ as the initial point for the Newton iteration instead of $u_0$ in the step 3 of Algorithm \ref{alg:NBI}. It is reasonable intuitively. Since $\|u_{k-1}-u_{k}\|\le M\|\lambda_{k-1}-\lambda_k\|\le M(b-a)/2^k$, where $M$ is the same as in Theorem \ref{thm:bisection}, $u_{k-1}$ sufficiently closes to $u_k$ to enter the quadratically convergence range faster as the bisection iteration proceeds. Although in this case, $u_k^0=u_{k-1}$ might not satisfy the condition \eqref{equ:condition}, we find in our numerical experiments that the algorithm can 'rectify' itself within one or two steps to guarantee the almost monotonicity.
\end{remark}
\end{remark}
At last, to complete the complexity analysis, let us discuss the cost of computing the linear system in step 6 of Algorithm \ref{alg:NBI}. If solving each linear system needs $T$ flops, the entire algorithm, with stopping tolerance as '$tol$', then needs approximately $K\cdot (l_0+log_2log_2(2c/tol))\cdot T$ flops.

Since the linear system in Algorithm \ref{alg:NBI} has a tridiagonal (or block tridiagonal) structure, it can be solved efficiently by both direct and iterative methods, see \cite{golub2013matrix}. In the following numerical experiments, we solve the system directly by the block tridiagonal LU factorization \cite{golub2013matrix} to obtain an 'exact' solution. Take $d=2$ in \eqref{equ:NEPV1} as an example, and suppose $u\in \mathbb{R}^n$ with $n=N^2$ in its discretization problem \eqref{equ:NEPv}. Then $B\in\mathbb{R}^{n\times n}$ is a block tridiagonal matrix with block dimension $N\times N$ and each block is of size $N\times N$. Then $T\approx N^3(7N/3+3)=O(n^2+n\sqrt{n})$.
\section{Numerical experiments}\label{sec:numerical}
In this section, we present some numerical results to illustrate our NBI algorithm. We compare NBI, NNI and the Matlab function 'fsolve'. All numerical experiments were performed on a Lenovo laptop with an Intel(R) Core(TM) Processor with access to 8GB of RAM using Matlab 2016b. In the first two examples below, we consider the finite difference discretization of GPE with $d=2$. $N$ is the number of discretization points except the ones in the bound along each direction. $n=N^2$ is the dimension of discretization eigenvector $u$. We also present an example with $d=3$ of small size in Example \ref{exa:example3}.
\subsection{Initial setting}
Firstly, an initial bisection interval $[a,b]$ is needed for NBI. Since our tests are to solve the discretization problem of \eqref{equ:NEPV1}, whose size becomes larger as the discretization gets finer and it will converges finally \cite{bao2013optimal}, we choose the interval as follows:
\begin{enumerate}[(1)]
\item For small size problem, let $a=\mu+eps$, where $\mu$ is the smallest eigenvalue of $B$, and $eps=10^{-3}$, or any appropriate small value. Then solve \eqref{equ:choi} with $\lambda$ given as $b=2a$ and check whether the solution $u$ satisfies $\|u\|>1$ or not. If not, let $a=b,~b=2a$ and repeat.
\item For problem of larger size, we use the same $[a,b]$ as the smaller one. Or we can use the eigenvalue of \eqref{equ:NEPv} computed from the problem with smaller size to set the interval heuristically.
\end{enumerate}

From our experiments, we find that as $\beta$ becomes larger, the positive eigenvalue of \eqref{equ:NEPv} increases, so the repeat times for finding $b$ also increases. However, for the problem of small size, it still can be computed quickly. Thus, we regard the $a,b$ as known constants.

Although \eqref{equ:condition} is required for the initial point $u_0$ of NBI theoretically, as noted in Remark \ref{rem:initial}, it can be relaxed and the algorithm still keep almost monotonic convergence. For convenient, we set $u_0=[1,\cdots,1]^T\in\mathbb{R}^{n}$ for NNI and 'fsolve'. Since the normlization is required by NNI, the initial point for NNI is $u_0=\frac{1}{\sqrt{n}}[1,\cdots,1]^T$.

For the Matlab function 'fsolve', we use the default settings except that $TolFun=10^{-10}$, which is an option in Matlab. To obtain an acceptable approximation solution, we set the stopping criterion for NBI as
\[|\|u\|-1|<10^{-7}.\]
The stopping criterion for NNI and the Newton iteration solving the subproblem of NBI is
\[\frac{\|u_k-u_{k-1}\|+\|\mathcal{A}(u_k)u_k-\lambda_ku_k\|}{\|u_k\|}<10^{-10}.\]
The maximum number of iterations allowed is 100 for NBI and NNI. For NNI, we use the Matlab function 'bicgstab' to solve the linear system with tolerance as $10^{-6}$ and maximum iteration number as $200$.
\begin{example}\label{exa:example1}
Consider the finite difference approximation with a grid size $h=1/(N+1)$ of \eqref{equ:NEPV1} with Dirichlet boundary conditions on $[0,1]\times[0,1]$, i.e.,
\begin{equation*}
(N+1)^2\beta\mathrm{diag}(u^{[2]})u+Bu=\lambda u,~u^Tu=1,
\end{equation*}
where $u\in\mathbb{R}^n$, $n=N^2$. $B=A+V$ where $A=I\otimes L_h+L_h\otimes I$ is a negative 2D Laplacian matrix with
\begin{equation*}
L_h=\frac{1}{h^2}
\begin{bmatrix}
2&-1& & &\\
-1&2&-1& &\\
&\ddots&\ddots&\ddots&\\
& &-1&2&-1\\
& & &-1&2
\end{bmatrix}
\in\mathbb{R}^{N\times N},
\end{equation*}
and $V={h^2}\mathrm{diag}(1^2+1^2,1^2+2^2,\cdots,N^2+N^2)$ is the discretization of the harmonic potential \cite{bao2003ground} $V(x,y)=x^2+y^2$.
\end{example}
For Example \ref{exa:example1}, Figure \ref{fig:quadratic} depicts how the residual $\mathcal{A}(u)u-\lambda u$ evolves versus the number of iterations for NNI and Newton iteration solving the unconstrained NEPv in NBI. Both of them indicate the quadratic convergence clearly. Table \ref{tab:NBIdetail} reports the number of iterations for each Newton iteration during the NBI. "Bi-Iter" denotes the number of outer bisection iterations to achieve convergence. "Newton-Iter" denotes the number of iterations for Newton iteration solving the subproblem in order and '3(8)' means that 3 repeats 8 times. "Violate" records the maximum number of violating the monotonicity in Newton iteration.
\begin{figure}
\includegraphics[width=.8\textwidth]{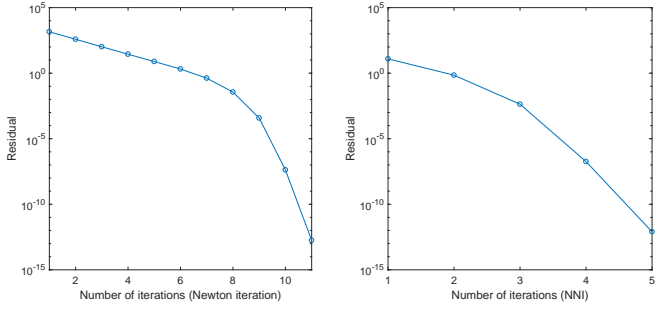}
\caption{The residual versus the number of iterations in the case $n=125,~\beta=1$. The left one is for Newton iteration with $\lambda=22.5$ and the right one is for NNI.}
\label{fig:quadratic}
\end{figure}
\begin{table}[htbp]
\caption{Numerical results for NBI when $\beta=1$, $[a,b]=[22,23]$.}
\label{tab:NBIdetail}
\begin{tabular}{|c|cccc|}
\hline
n     & Bi-Iter & Newton-Iter               & Violate & Residual   \\ \hline
225   & 18      & {[}12,5,4(4),3(8),2(4){]} & 2       & 1.7600e-13 \\
3639  & 16      & {[}15,5,4(4),3(8),2(2){]} & 2       & 3.7432e-12 \\
16129 & 17      & {[}17,5,4(4),3(8),2(3){]} & 2       & 1.6660e-11 \\ \hline
\end{tabular}
\end{table}
From the table, we see that the number of iterations that Newton iteration needs decrease versus the iteration for bisection, indicating that the solution for last bisection iteration step has fallen into the quadratic convergence range for the current Newton iteration. In all of our tests, the Newton iteration needs no more than 5 steps to convergence except the one for the first bisection iteration, it implies we have chosen a conservative initial point. And each Newton iteration still converges monotonicity except at most two steps. Figure \ref{fig:tol} and \ref{fig:bound} show the number of both outer and total iterations versus the magnitude of the stopping tolerance '$tol$' for bisection iteration and the interval. As we see, the number of iterations approximately increases linearly as the logarithm of the magnitude of $1/tol$ and $(b-a)$, which is consistent with our complexity analysis.
\begin{figure}
\includegraphics[width=.8\textwidth]{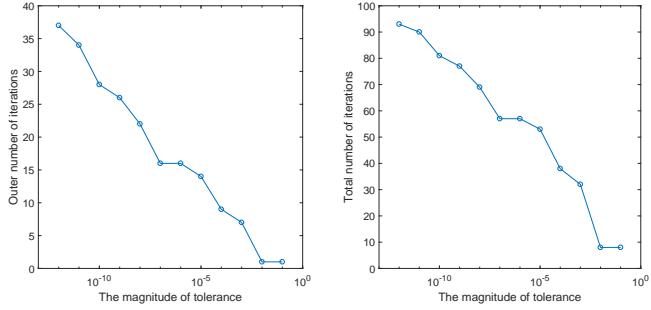}
\caption{The number of iterations versus the magnitude of tolerance for $\beta=1$, $n=3639$, $[a,b]=[22,23]$.}
\label{fig:tol}
\end{figure}
\begin{figure}
\includegraphics[width=.8\textwidth]{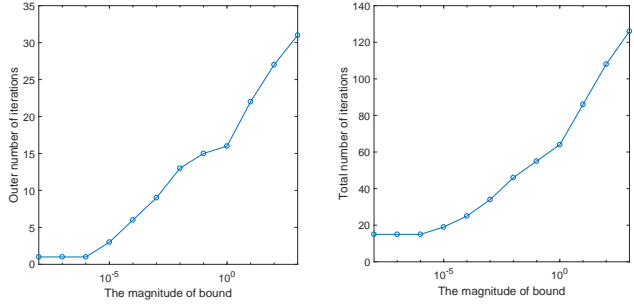}
\caption{The number of iterations versus the magnitude of bound for $\beta=1$, $n=3639$.}
\label{fig:bound}
\end{figure}

Table \ref{tab:compare1} reports the results obtained by NBI, NNI and fsolve. In the table, 'Iter' denotes the number of outer iterations, $[a,b]$ is the initial interval for NBI. When $n=127^2$, fsolve needs much more time than NBI and NNI to stop, we use '-' to represent this case. We find that both NBI and NNI are efficient compared with fsolve. Note that although NBI needs more iterations, solving the subproblem in each bisection iteration is easy and rather efficient and thus makes it more competitive particularly for relatively large $n$. We also find that although 'bicgstab' usually cannot obtain a solution of the linear system that meets the accuracy requirement within maximum iteration steps 200, NNI still works quite well and converges.
\begin{table}[H]
\caption{Numerical results for NBI, NNI and fsolve.}
\label{tab:compare1}
\begin{center}
\begin{tabular}{|c|cccc||cccc|}
\hline
solver & Iter                    & CPU(s)                       & Residual                        & $\lambda$ & Iter                    & CPU(s)                        & Residual                        & $\lambda$ \\ \hline
       & \multicolumn{8}{c|}{$\beta=50,~[a,b]=[80.9598,161.9196]$}                                                                                                                                                    \\ \hline
       & \multicolumn{4}{c||}{$n=15^2$}                                                                        & \multicolumn{4}{c|}{$n=31^2$}                                                                         \\ \hline
NBI    & \multicolumn{1}{c|}{19} & \multicolumn{1}{c|}{0.1264}  & \multicolumn{1}{c|}{2.4320e-13} & 100.4052  & \multicolumn{1}{c|}{21} & \multicolumn{1}{c|}{1.2969}   & \multicolumn{1}{c|}{8.7800e-13} & 100.8487  \\ \hline
NNI    & \multicolumn{1}{c|}{5}  & \multicolumn{1}{c|}{0.1063}  & \multicolumn{1}{c|}{2.7287e-13} & 100.4052  & \multicolumn{1}{c|}{6}  & \multicolumn{1}{c|}{0.3851}   & \multicolumn{1}{c|}{1.3404e-12} & 100.8487  \\ \hline
fsolve & \multicolumn{1}{c|}{22} & \multicolumn{1}{c|}{0.4588}  & \multicolumn{1}{c|}{1.5143e-13} & 100.4052  & \multicolumn{1}{c|}{16} & \multicolumn{1}{c|}{2.8695}   & \multicolumn{1}{c|}{2.8681e-13} & 100.8487  \\ \hline
       & \multicolumn{4}{c||}{$n=63^2$}                                                                        & \multicolumn{4}{c|}{$n=127^2$}                                                                        \\ \hline
NBI    & \multicolumn{1}{c|}{21} & \multicolumn{1}{c|}{15.6394} & \multicolumn{1}{c|}{4.0331e-12} & 100.9569  & \multicolumn{1}{c|}{22} & \multicolumn{1}{c|}{253.7449} & \multicolumn{1}{c|}{1.8729e-11} & 100.9838  \\ \hline
NNI    & \multicolumn{1}{c|}{6}  & \multicolumn{1}{c|}{10.1329} & \multicolumn{1}{c|}{4.6803e-12} & 100.9569  & \multicolumn{1}{c|}{6}  & \multicolumn{1}{c|}{265.3062} & \multicolumn{1}{c|}{2.7074e-11} & 100.9838  \\ \hline
fsolve & \multicolumn{1}{c|}{18} & \multicolumn{1}{c|}{41.3312} & \multicolumn{1}{c|}{1.1446e-12} & 100.9569  & \multicolumn{1}{c|}{-}  & \multicolumn{1}{c|}{-}        & \multicolumn{1}{c|}{-}          & -         \\ \hline
       & \multicolumn{8}{c|}{$\beta=100,~[a,b]=[166,170]$}                                                                                                                                                            \\ \hline
       & \multicolumn{4}{c||}{$n=15^2$}                                                                        & \multicolumn{4}{c|}{$n=31^2$}                                                                         \\ \hline
NBI    & \multicolumn{1}{c|}{16} & \multicolumn{1}{c|}{0.0788}  & \multicolumn{1}{c|}{2.9063e-13} & 166.0699  & \multicolumn{1}{c|}{16} & \multicolumn{1}{c|}{0.9340}   & \multicolumn{1}{c|}{1.0821e-12} & 167.0551  \\ \hline
NNI    & \multicolumn{1}{c|}{6}  & \multicolumn{1}{c|}{0.0231}  & \multicolumn{1}{c|}{3.8249e-13} & 166.0699  & \multicolumn{1}{c|}{6}  & \multicolumn{1}{c|}{0.3695}   & \multicolumn{1}{c|}{1.1576e-12} & 167.0551  \\ \hline
fsolve & \multicolumn{1}{c|}{14} & \multicolumn{1}{c|}{0.2504}  & \multicolumn{1}{c|}{1.1241e-13} & 166.0699  & \multicolumn{1}{c|}{16} & \multicolumn{1}{c|}{2.8015}   & \multicolumn{1}{c|}{3.2640e-13} & 167.0551  \\ \hline
       & \multicolumn{4}{c||}{$n=63^2$}                                                                        & \multicolumn{4}{c|}{$n=127^2$}                                                                        \\ \hline
NBI    & \multicolumn{1}{c|}{15} & \multicolumn{1}{c|}{11.1967} & \multicolumn{1}{c|}{4.0381e-12} & 167.2938  & \multicolumn{1}{c|}{16} & \multicolumn{1}{c|}{177.1247} & \multicolumn{1}{c|}{1.8998e-11} & 167.3528  \\ \hline
NNI    & \multicolumn{1}{c|}{6}  & \multicolumn{1}{c|}{9.0215}  & \multicolumn{1}{c|}{6.4312e-12} & 167.2938  & \multicolumn{1}{c|}{6}  & \multicolumn{1}{c|}{258.9604} & \multicolumn{1}{c|}{2.3271e-11} & 167.3528  \\ \hline
fsolve & \multicolumn{1}{c|}{18} & \multicolumn{1}{c|}{40.8353} & \multicolumn{1}{c|}{1.2025e-12} & 167.2938  & \multicolumn{1}{c|}{-}  & \multicolumn{1}{c|}{-}        & \multicolumn{1}{c|}{-}          & -         \\ \hline
\end{tabular}
\end{center}
\end{table}
\begin{example}\label{exa:example2}
Consider the problem defined as in Example \ref{exa:example1}, with the combined harmonic and optical lattice potential \cite{bao2006efficient}
\[V(x,y)=\frac{1}{2}(x^2+y^2)+50[\sin^2(\frac{\pi x}{4})+\sin^2(\frac{\pi y}{4})].\]
\end{example}
Table \ref{tab:compare2} reports the numerical results for NBI, NNI and fsolve. The notation is the same as Table \ref{tab:compare1}.
\begin{table}[H]
\caption{Numerical results for NBI, NNI and fsolve}\label{tab:compare2}
\begin{center}
\begin{tabular}{|c|cccc||cccc|}
\hline
solver & Iter                    & CPU(s)                       & Residual                        & $\lambda$ & Iter                    & CPU(s)                        & Residual                        & $\lambda$ \\ \hline
       & \multicolumn{8}{c|}{$\beta=1,~[a,b]=[34.4188,68.8377]$}                                                                                                                                                                               \\ \hline
       & \multicolumn{4}{c||}{$n=63^2$}                                                                        & \multicolumn{4}{c|}{$n=127^2$}                                                                        \\ \hline
NBI & \multicolumn{1}{c|}{23} & \multicolumn{1}{c|}{19.0777} & \multicolumn{1}{c|}{3.6937e-12} & 36.9082   & \multicolumn{1}{c|}{24} & \multicolumn{1}{c|}{297.5618} & \multicolumn{1}{c|}{1.7121e-11} & 36.9121   \\ \hline
NNI    & \multicolumn{1}{c|}{6}  & \multicolumn{1}{c|}{13.4127} & \multicolumn{1}{c|}{6.2318e-12} & 36.9082   & \multicolumn{1}{c|}{13} & \multicolumn{1}{c|}{576.9285} & \multicolumn{1}{c|}{9.7777e-11} & 36.9121   \\ \hline
fsolve & \multicolumn{1}{c|}{18} & \multicolumn{1}{c|}{43.7313} & \multicolumn{1}{c|}{1.5019e-12} & 36.9082   & \multicolumn{1}{c|}{-}  & \multicolumn{1}{c|}{-}        & \multicolumn{1}{c|}{-}          & -         \\ \hline
       & \multicolumn{8}{c|}{$\beta=50,~[a,b]=[68.8377,137.6753]$}                                                                                                                                                                              \\ \hline
       & \multicolumn{4}{c||}{$n=63^2$}                                                                        & \multicolumn{4}{c|}{$n=127^2$}                                                                        \\ \hline
NBI & \multicolumn{1}{c|}{21} & \multicolumn{1}{c|}{15.3934} & \multicolumn{1}{c|}{4.0629e-12} & 117.4751  & \multicolumn{1}{c|}{17} & \multicolumn{1}{c|}{219.2103} & \multicolumn{1}{c|}{1.8505e-11} & 117.5013  \\ \hline
NNI    & \multicolumn{1}{c|}{6}  & \multicolumn{1}{c|}{10.7667} & \multicolumn{1}{c|}{5.9779e-12} & 117.4751  & \multicolumn{1}{c|}{7}  & \multicolumn{1}{c|}{312.3404} & \multicolumn{1}{c|}{2.6694e-11} & 117.5013  \\ \hline
fsolve & \multicolumn{1}{c|}{18} & \multicolumn{1}{c|}{41.9088} & \multicolumn{1}{c|}{1.2115e-12} & 117.4751  & \multicolumn{1}{c|}{-}  & \multicolumn{1}{c|}{-}        & \multicolumn{1}{c|}{-}          & -         \\ \hline
       & \multicolumn{8}{c|}{$\beta=100,~[a,b]=[137.6753,275.3506]$}                                                                                                                                                                             \\ \hline
       & \multicolumn{4}{c||}{$n=63^2$}                                                                        & \multicolumn{4}{c|}{$n=127^2$}                                                                        \\ \hline
NBI & \multicolumn{1}{c|}{22} & \multicolumn{1}{c|}{16.0668} & \multicolumn{1}{c|}{4.4592e-12} & 184.1856  & \multicolumn{1}{c|}{21} & \multicolumn{1}{c|}{222.7564} & \multicolumn{1}{c|}{1.9445e-11} & 184.2434  \\ \hline
NNI    & \multicolumn{1}{c|}{6}  & \multicolumn{1}{c|}{9.2104}  & \multicolumn{1}{c|}{6.4338e-12} & 184.1856  & \multicolumn{1}{c|}{6}  & \multicolumn{1}{c|}{265.9054} & \multicolumn{1}{c|}{2.6092e-11} & 184.2434  \\ \hline
fsolve & \multicolumn{1}{c|}{18} & \multicolumn{1}{c|}{41.2992} & \multicolumn{1}{c|}{1.5302e-12} & 184.1856  & \multicolumn{1}{c|}{-}  & \multicolumn{1}{c|}{-}        & \multicolumn{1}{c|}{-}          & -         \\ \hline
       & \multicolumn{8}{c|}{$\beta=1000,~[a,b]=[1.1014e3,2.2028e3]$}                                                                                                                                                                            \\ \hline
       & \multicolumn{4}{c||}{$n=63^2$}                                                                        & \multicolumn{4}{c|}{$n=127^2$}                                                                        \\ \hline
NBI & \multicolumn{1}{c|}{19} & \multicolumn{1}{c|}{14.7042} & \multicolumn{1}{c|}{5.6476e-12} & 1.2053e3  & \multicolumn{1}{c|}{19}   & \multicolumn{1}{c|}{213.0611}       & \multicolumn{1}{c|}{2.0743e-11}   & 1.2065e3  \\ \hline
NNI    & \multicolumn{1}{c|}{6}  & \multicolumn{1}{c|}{5.6909}  & \multicolumn{1}{c|}{6.7445e-12} & 1.2053e3  & \multicolumn{1}{c|}{6}   & \multicolumn{1}{c|}{135.6132}        & \multicolumn{1}{c|}{3.1946e-11}   & 1.2065e3  \\ \hline
fsolve & \multicolumn{1}{c|}{19} & \multicolumn{1}{c|}{43.1263} & \multicolumn{1}{c|}{1.4026e-12} & 1.2053e3  & \multicolumn{1}{c|}{-}  & \multicolumn{1}{c|}{-}        & \multicolumn{1}{c|}{-}          & -         \\ \hline
\end{tabular}
\end{center}
\end{table}
\begin{example}\label{exa:example3}
Consider the finite difference approximation of \eqref{equ:NEPV1} with $d=3$ and Dirichlet boundary conditions on $[0,1]\times [0,1]\times [0,1]$. The discretization is as follows,
\begin{equation*}
(N_x+1)(N_y+1)(N_z+1)\beta\mathrm{diag}(u^{[2]})u+Bu=\lambda u,~u^Tu=1,
\end{equation*}
where $u\in\mathbb{R}^n$, $n=N_xN_yN_z$, and $N_x,N_y,N_z$ are numbers of split points along each direction except for endpoints. $B=A+V$ where $A=I_{N_z}\otimes I_{N_y}\otimes L_{h_x}+I_{N_z}\otimes L_{h_y}\otimes I_{N_x}+L_{h_z}\otimes I_{N_y}\otimes I_{N_x}$, where $I_{N}$ is the $N\times N$ matrix and $L_h$ is defined as Example \ref{exa:example1}. $h_x=1/(N_x+1)$ and $h_y$, $h_z$ are defined similarly. $V$ is the discretization of $V(x,y,z)=x^2+y^2+z^2$. See Table \ref{tab:3d} for the numerical results of the three-dimensional example.
\end{example}
\begin{table}[H]
\caption{Numerical results for NBI, NNI and fsolve when $\beta=1$.}\label{tab:3d}
\begin{center}
\begin{tabular}{|c|c|c|c|c||c|c|c|c|}
\hline
solver & Iter  & CPU(s)    & Residual    & $\lambda$ & Iter  & CPU(s)    & Residual    & $\lambda$ \\ \hline
       & \multicolumn{4}{c||}{$N_x=N_y=17$, $N_z=33$} & \multicolumn{4}{c|}{$N_x=17$, $N_y=N_z=33$} \\ \hline
NBI    & 20    & 62.3397   & 6.5372e-13  & 19.7394   & 18    & 223.1240  & 1.1629e-12  & 19.7574   \\ \hline
NNI    & 6     & 25.3299   & 7.0243e-13  & 19.7394   & 6     & 119.1262  & 9.0901e-13  & 19.7574   \\ \hline
fsolve & 16    & 171.6908  & 7.3194e-13  & 19.7394   & -     & -         & -           &  -         \\ \hline
\end{tabular}
\end{center}
\end{table}
\section{Conclusion}
In this paper, we are concentrated with the positive eigenpair of the nonlinear eigenvalue problem with eigenvector nonlinearity (NEPv) generated by the finite difference discretization of the GPE. The Newton-Noda iteration \cite{liu2020positivity} preserving the positivity was transferred to GPE. We observed that it still guarantees the locally quadratic convergence even though the inner linear system is solved inexactly, which may need further research. We proposed the Newton-Bisection method, which is easy to implement. We then gave the computational complexity analysis for the Newton-Bisection method in detail, which gives an intuitive explanation for the phenomena that the convergence is getting faster as the bisection iteration proceeds. Another advantage of Newton-Bisection method is the easy-to-solve subproblems in each bisection iteration. With the nice block tridiagonal structure, we used more effective strategy solving subproblems.


%
 \section*{Conflict of interest}

 The authors declare that they have no conflict of interest.

\bibliographystyle{spmpsci}      
\bibliography{ref}   


\end{document}